%% file: main.tex
\documentclass[12 pt]{article}
\input{header}

\title{Closed Neighborhood Balanced $k$-Coloring of Graphs}

\authorentry{Maurice Genevieva Almeida}{}{p20230078@goa.bits-pilani.ac.in}{}{}
 \authorentry{Ravindra Pawar }{2}{}{0000-0003-4541-8462}{https://sites.google.com/view/ravi-pawar}
\authorentry{Siddharth Gupta}{3}{siddharthg@goa.bits-pilani.ac.in}{0000-0003-4671-9822}{https://guptasid.bitbucket.io/}
\authorentry{Tarkeshwar Singh}{4}{tksingh@goa.bits-pilani.ac.in}{}{}

\affilentry{1,4}{Department of Mathematics, BITS Pilani, K K Birla Goa Campus, Goa, India.}
 \affilentry{3}{Department of Computer Science and Information Systems, BITS Pilani, K K Birla Goa Campus, Goa, India.}
\affilentry{2}{The Institute of Mathematical Sciences, HBNI, Chennai, India.}

\begin{document}

\date{}
\maketitle
\begin{abstract}
\input{abstract.tex}
\end{abstract}

\noindent \textbf{2020 Mathematics Subject Classification:} 05C 15, 05C 78. \\
\textbf{Keywords:} graph coloring, \cnbc.

\section{Introduction}
\input{introduction.tex}

%\section{Known Results}
%\input{knownresults.tex}

\section{Results} \label{sec:main}

\input{results.tex}
\section{Hardness Results for \protect{$k \geq 3$}}\label{sec:Hardness,kgeq3}
\input{Hardness_k_geq3}

%\section{Hardness Results} \label{sec:hardness}
%\input{hardness_results.tex}

% \section{More Results}
% \input{more results.tex}

\section{Conclusion}
\input{conclusion.tex}

%\appendix\label{appendix}
%\section{Correction to \texorpdfstring{\NP}{NP}-hardness proof in \cite{nbc-np}}
%\input{appendix1.tex}

\bibliography{bibliography.bib}

\end{document}

%% file: header.tex
\usepackage[letterpaper,margin=1in]{geometry}
\usepackage[utf8]{inputenc}
\usepackage[english]{babel}
\usepackage{enumitem}
\usepackage{ifthen}
\usepackage{subcaption}
\usepackage[pagewise]{lineno} 
\usepackage[table]{xcolor}
\usepackage[colorlinks=true,citecolor=blue,linkcolor=blue, urlcolor=gray]{hyperref}
\usepackage{fontawesome5}
\usepackage{xparse}       % for optional argument parsing
\usepackage{etoolbox}     % for \g@addto@macro
\usepackage{authblk}      % optional, can be skipped if doing custom title
\usepackage{xspace}
\usepackage{multirow}
\usepackage[font=small]{caption}

\makeatletter
\let\@authorline\@empty    % inline authors
\let\@affilblock\@empty    % numbered affiliations
\newcommand{\emailicon}[1]{\href{mailto:#1}{\textsuperscript{\faEnvelope[regular]}}}
\newcommand{\orcidicon}[1]{\href{https://orcid.org/#1}{\textsuperscript{\faOrcid}}}
\newcommand{\homepage}[1]{\href{#1}{\textsuperscript{\faHome}}}
\newcommand{\authorentry}[5]{%
  \ifx\@authorline\@empty
  \else
    \g@addto@macro\@authorline{,\ }%
  \fi
  \g@addto@macro\@authorline{%
    #1%
    \ifx&#2&\else\textsuperscript{#2}\fi%
    \ifx&#3&\else\emailicon{#3}\fi%
    \ifx&#4&\else\orcidicon{#4}\fi%
    \ifx&#5&\else\homepage{#5}\fi%
  }%
}
\newcommand{\affilentry}[2]{%
  \g@addto@macro\@affilblock{%
    \textsuperscript{#1}#2\\%
  }%
}
\renewcommand{\maketitle}{%
  \begin{center}
    {\LARGE \bfseries \@title \par}
    \vskip 1em
    {\normalsize
      \@authorline
    }
    \vskip 1em
    {\small \@affilblock}
    \vskip 1em
    {\small \@date \par}
  \end{center}
  \vskip 2em
}
\makeatother

\usepackage{amsmath}
\usepackage{amsfonts}
\usepackage{amssymb}
\usepackage{amsthm}
\usepackage{thmtools}
\usepackage{complexity}
\usepackage[capitalize,nameinlink]{cleveref} % the hyperlink is only on the number by default. [nameinlink] option includes "Theorem", "Figure", etc in the hyperlink
\crefname{equation}{Equation}{Equations}
\crefname{figure}{Figure}{Figures}
\newtheorem{theorem}{Theorem}[section]
\newtheorem{proposition}[theorem]{Proposition}
\newtheorem{lemma}[theorem]{Lemma}
\newtheorem{corollary}[theorem]{Corollary}

\newtheorem{definition}[theorem]{Definition}
\newtheorem{observation}[theorem]{Observation}

\newtheorem{example}[theorem]{Example}
\usepackage{float}
\usepackage{tikz}
\usetikzlibrary{math}
\usetikzlibrary{calc}
\usetikzlibrary{positioning}
\tikzstyle{vtx}=[circle, draw, fill=black, inner sep=0pt, minimum width=5pt]
\tikzstyle{vtx-white}=[circle, draw, fill=white, inner sep=0pt, minimum width=5pt]

\newcommand{\nbc}{neighborhood-balanced $k$-coloring\xspace}

\newcommand{\nbcd}{neighborhood-balanced $k$-colored\xspace}
\newcommand{\cnbc}{closed neighborhood balanced $k$-coloring\xspace}
\newcommand{\cnbcl}{closed neighborhood balanced $k$-colorable\xspace}

\newcommand{\cnbcd}{closed neighborhood balanced $k$-colored\xspace}

%% file: abstract.tex
For a simple graph $G = (V, E)$ and a positive integer $k \ge 2$, a coloring of vertices of $G$ using exactly $k$ colors such that every vertex has an equal number of vertices of each color in its closed neighborhood is called \emph{closed neighborhood-balanced $k$-coloring}, and the graph which admits such a coloring is called \cnbcd graph.
%This generalizes the notion of neighborhood balanced coloring of graphs introduced by Bryan Freyberg and Alison Marr (Graphs and Combinatorics, 2024).
We derive some necessary/sufficient conditions for a graph to admit a \cnbc and discuss various graph operations involving such graphs.
Furthermore, we prove that there is no forbidden subgraph characterization for the class of closed neighborhood-balanced $ k$-colorable graphs.

%% file: introduction.tex
Let $G=(V,E)$ be a simple graph.
%The \textit{order} of $G$ is $|V|$, and the \textit{size} of $G$ is $|E|$, and 
For any vertex $v\in V$, define an \textit{ open neighborhood} of $v$ as a set $N(v):=\{u: uv\in E\}$.
The members of $N(v)$ are called the \textit{neighbors} of $v$ and the cardinality of $N(v)$ is called the \emph{degree} of $v$, denoted as $d(v)$.
Further, if we add the vertex $v$ to its open neighborhood $N(v)$, then we get the closed neighborhood of $v$, denoted by $N[v]$.
For graph-theoretic notation, we refer to Chartrand and Lesniak \cite{chart}.

% Graph labeling problems lie at the intersection of combinatorial optimization and structural graph theory.
% It offers a framework for assigning algebraic or numeric values to graph elements under specific constraints.
% For example, a proper vertex coloring of a graph can be considered as a labeling of the vertices of $G$ using numbers such that the ends of each edge are labeled differently.
% In this field of graph theory, there are a number of open problems and several challenging conjectures like the `Graceful Tree Conjecture' and the `Antimagic Graph Conjecture' \cite{gallian}.

Freyberg et al. \cite{nbc} introduced the concept `neighborhood balanced coloring'.
A \emph{neighborhood balanced coloring} of a graph $G$ is a vertex coloring of $G$ using two colors, say red and green, such that each vertex has an equal number of neighbors of both colors.
%A graph that admits a neighborhood balanced coloring is called \emph{neighborhood balanced colorable}.
It is easy to see that if a graph admits a neighborhood balanced coloring, then the degree of every vertex is even.
This notion of coloring is somewhat similar to \emph{cordial labeling} \cite{cordial}.
In \cite{nbc}, the authors presented several necessary/sufficient conditions for a graph to admit neighborhood balanced coloring. They further presented several graph classes that admit neighborhood balanced coloring.

Minyard et al. \cite{3nbc} introduced the concept of `neighborhood balanced $3$-coloring' of a graph and gave a characterization of several classes of graphs that admit such a coloring.
A graph is said to admit a \emph{neighborhood balanced $3$-coloring} if the vertices of $G$ can be colored using exactly three colors such that each vertex has an equal number of neighbors of all three colors.

Almeida et al. \cite{knbc} generalized this concept by introducing the notion of neighborhood balanced $k$-coloring of graphs, for any positive integer $k \ge 2$. A coloring of the vertices of $G$ using $k$ colors such that every vertex has an equal number of vertices of each color in its neighborhood is called neighborhood-balanced $k$-coloring, and the graph is called \nbcd graph.

Motivated by the notion of neighborhood balanced colorings of graphs, 
Collins et al.~\cite{collin-cnbc} introduced the concept of 
\emph{closed neighborhood balanced coloring}. 
In this setting, the vertices of a graph are colored with two colors, say red and green, such that in the closed neighborhood of every vertex, the number of red vertices equals the number of green vertices. 
Collins et al.~\cite{collin-cnbc} established several necessary/sufficient conditions for a graph to admit a closed neighborhood balanced coloring. 
They identified various classes of graphs that admit such a coloring and further investigated the behavior of closed neighborhood balanced colored 
graphs under different graph operations. Moreover, they demonstrated that the class of closed neighborhood balanced colored graphs is not hereditary.

In this article, we study  the generalized version of closed neighborhood balanced coloring coloring using $k$ colors, called \emph{\cnbc} of graphs, for any positive integer $k \ge 2$.

Formally, a \cnbc of a graph is defined as follows.
\begin{definition} \label{def:nbc}
    Let $G$ be a graph, and let $k \ge 2$ be an integer.
    If the vertices of $G$ can be colored using the $k$ colors, say $1, 2,\dots, k$, so that every vertex has an equal number of vertices of each color in its closed neighborhood, then the coloring is called a \emph{\cnbc of $G$.}
    A graph that admits such a coloring is called a \emph{\cnbcl graph or \cnbcd graph.} 
\end{definition}

In other words, a \cnbc of graph $G$ is a partition of the vertex set $V$ of $G$ into $k$ sets $V_1, V_2, \dots, V_k$ (that is, color classes) such that every vertex has an equal number of vertices from each set in its closed neighborhood.
Note that if $c$ is a \cnbc of graph $G$ using colors in the order $(1, 2, \dots, k)$, then the colorings $c=c_1, c_2, \dots, c_k$ obtained by rotating the colors in a cyclic order are also closed neighborhood-balanced $k$-colorings of $G$.
We refer to such colorings as the colorings obtained by cyclic shift of $c$.

In this article, we begin by proving the fact that the class of \cnbcd graphs is not hereditary.
Following this, we present some necessary conditions for graphs to admit a \cnbc.
We also study closed neighborhood balanced colored graphs under various graph operations.
We begin by citing the following two theorems.

\begin{lemma}\cite{knbc}\label{nbclemma}
If a graph $G$ admits a \nbc, then the degree of every vertex is a multiple of $k$. 
\end{lemma}

\begin{theorem}\cite{knbc}\label{lexiknbc}
    Let $G$ and $H$ be two graphs. If $H$ admits a \nbc $c$ with $|V^c_i(H)|=|V^c_j(H)|$, then the lexicographic product $G[H]$ admits a \nbc.
\end{theorem}

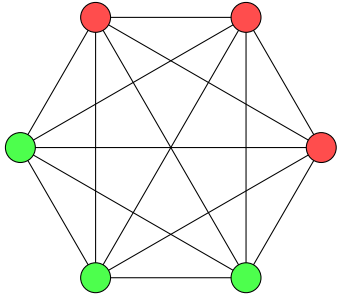
\begin{figure}
 \centering
 \input{K6.tikz}
\caption{A closed neighborhood-balanced $2$-coloring of $K_{6}$.}
\label{fig: cnbc of K6}
 \end{figure}

%% file: K6.tikz
\begin{tikzpicture}[scale=2]

  % --- Vertex coordinates (manual hexagon layout) ---
  \coordinate (v1) at (1.000, 0.000);
  \coordinate (v2) at (0.500, 0.866);
  \coordinate (v3) at (-0.500, 0.866);
  \coordinate (v4) at (-1.000, 0.000);
  \coordinate (v5) at (-0.500,-0.866);
  \coordinate (v6) at (0.500,-0.866);

  % --- Edges of K6 ---
  \foreach \i/\j in {1/2,1/3,1/4,1/5,1/6,
                     2/3,2/4,2/5,2/6,
                     3/4,3/5,3/6,
                     4/5,4/6,
                     5/6}{
    \draw (v\i) -- (v\j);
  }

  % --- Vertices with customizable colors ---
  % Change 'fill=<color>' to set vertex colors
  \node[draw, circle, fill=red!70,   inner sep=4pt] at (v1) {};
  \node[draw, circle, fill=red!70, inner sep=4pt] at (v2) {};
  \node[draw, circle, fill=red!70,  inner sep=4pt] at (v3) {};
  \node[draw, circle, fill=green!70,inner sep=4pt] at (v4) {};
  \node[draw, circle, fill=green!70,inner sep=4pt] at (v5) {};
  \node[draw, circle, fill=green!70,inner sep=4pt] at (v6) {};

\end{tikzpicture}

%% file: results.tex
We begin this section by deriving some necessary conditions for graphs to admit a \cnbc.
Note that if a graph admits such a coloring, then the degree of every vertex should be  at least $k-1$.
We state this necessary condition in the following lemma in a more general form.

\begin{lemma}\label{cnbclem1}
If a graph $G$ admits a \cnbc, then for any vertex $v$ of $G$, $d(v)\equiv -1\pmod{k}$.
\end{lemma}
The next observation talks about vertices having the same neighborhood in a graph.
\begin{observation}
    Let $G$ be a graph and $u,v\in V(G)$ such that $N(u)=N(v)$. Then $u$ and $v$ are of the same color in any \cnbc of $G$.
\end{observation}

Observe that a \cnbcd graph without isolated vertices should have at least one vertex of every color.

\begin{lemma} \label{th:order of nbc graph}
    If $G$ is a \cnbcd graph without isolated vertices, then the order of $G$ is at least $k$.
\end{lemma}
The lower bound on the order of a \cnbcd graph given in the above lemma is sharp.
There are \cnbcd graphs of order $k$, for example, the complete graph $K_{k}$, where each of the $k$ vertices receives one of the $k$ colors.

The following lemma shows that if a graph $G$ is \cnbcl, where $k$ is even, then it is also closed neighborhood-balanced $2$-colorable.
However, by Lemma \ref{cnbclem1}, the converse need not be true.

\begin{lemma}\label{cnbc:reductionlemma}
If a graph $G$ admits a \cnbc and $p$ divides $k$, then the graph $G$ also admits a closed neighborhood-balanced $p$-coloring.
\end{lemma}
\begin{proof}
   Let \( f \colon V(G) \to \{1, 2, \dots, k\} \) be a \cnbc\ of \(G\). Define a new coloring by replacing each color \( i \in \{1, 2, \dots, k\} \) with the color \( j \in \{1, 2, \dots, p\} \), where \(j \equiv i (\bmod\ p )\). Since \( p \) divides \( k \), each of the \( p \) colors replaces exactly the same number of original colors. Moreover, because the closed neighborhood of every vertex in \(G\) has an equal number of vertices in each of the original \( k \) color classes, it follows that the closed neighborhood of every vertex still has an equal number of vertices in each of the \( p \) new color classes.
\end{proof}
Using Lemma \ref{cnbc:reductionlemma}, we can say that the complete graph $K_n$ admits \cnbc if and only if $n\equiv 0 (\bmod\ k)$. We now give a complete characterization of \cnbcl hamming graphs.
We first give a formal definition of hamming graphs.
For a set $S$, define $S^{d} := \underbrace{S \times S \times \cdots \times S}_{d-\mathrm{times}}$ to be the $d$-fold cartesian product of $S$ with itself.
A \emph{hamming} graph, denoted as $H(d,k)$, is a graph with vertex set $S^{d}$, where $|S| = k$, and two vertices are adjacent if they differ in exactly one coordinate.
Note that the hamming $H(d,k)$ is $d(k-1)$-regular.
For our convenience, without loss of generality, we take $S = \{1, 2, \dots, k\}$.

\begin{theorem} \label{th: hamming graphs}
The hamming graph $H(d,k)$ is a closed \cnbcd graph if and only if $d\equiv 1 (\bmod\ k)$.
\end{theorem}

\begin{proof}
If $d\not\equiv (\bmod\ k)$, then the degree of any vertex of $H(d,k)$ is not congruent to $-1(\bmod\ k)$ and hence it is not a \cnbcd graph.
Conversely, suppose that $d \equiv 1 (\bmod\ k)$. 
Write $d = kn+1$.
We show that $H(kn+1,k)$ admits a \cnbc by induction on $n$.
\par
For $n=1$, consider the hamming graph $H(k+1,k)$.
Each vertex of $H(k+1,k)$ is represented as a $(k+1)$-tuple $(a_1, a_2, \dots, a_k,a_{k+1})$ where each $a_i \in S$.
For each $(a_1, a_2, \dots, a_{k}) \in S^{k}$, consider set $X^{1}_{(a_1, \dots, a_{k})} = \{(a_1, a_2, \dots, a_{k}, y_{k+1}) : y_{k+1} \in S\}$.
That is, each $X^1$ is a collection of vertices having the same first $k$ coordinates.
There are $k^{k}$ such sets, each of cardinality $k$, and it gives us a partition of $S^{k+1} = V(H)$ (refer to Table \ref{tab: partition of vertex set of H(4,4)}).
Again for each $(a_1, a_2,  \dots, a_{k-1}) \in S^{k-1}$, let $X^{2}_{(a_1,\dots, a_{k-1})} : = \{(a_1, a_2, \dots,a_{k-1}, y_{k}, y_{k+1}) : y_{k},\ y_{k+1} \in S\}$.
That is, each $X^2$ is a collection of vertices having the same first $k-1$ coordinates.
There are $K^{k-1}$ such sets, each of cardinality $k^2$, and it gives a partition of the collection of all sets of the form $X^1$.
Continuing in this way, for each $r = 1, \dots, k$, we obtain the set $X^{r}_{(a_1,a_2,\dots,a_{k+1-r})} =\{ (a_1, a_2,\dots, a_{k+1-r}, y_{k+2-r}, \dots, y_{k+1}) : y_{k+2-r}, \dots, y_{k+1} \in S\}$ having the same first $k+1-r$ coordinates.
Note that for each $r = 1, \dots, k$, $|X_{(a_1,a_2,\dots,a_{k+1-r})}^{r}| = (k+1)^r$ and it gives us a partition of all sets of the form $X^{r-1}$ (refer to Table \ref{tab: partition of vertex set of H(4,3)}).
It is straightforward to verify that for each $1 \le i \le k$, the set $X^i$ consists of $k$ sets of type $X^{i-1}$.
Note that for each $1 \le j \le k$, the tuple in the suffix of the sets of the type $X^j$ is of length $k+1-j$.
\par
Now we give the coloring scheme.
Consider $(1, 1,\dots, 1) \in S^{k+1}$.
We color the $k$ vertices in $X^1_{(1,1, \dots, 1)}$ by the coloring $c^{1}$ so that it is a rainbow set.
Then for $i \in S \setminus \{1\}$, we color the vertices of the form $X^1_{(1,1, \dots, 1, i)}$ by $c^{1}_{i}$ obtained by $i$th cyclic shift of colors in $c^{1}$ (recall that by our convention, the first cyclic shift of $c^1$ is $c^1$ itself).
This gives a coloring of $X^2_{(1,1, \dots, 1)}$.
We call this coloring scheme $c^2$.
Again for each $i \in S \setminus \{1\} $, we color $k^2$ vertices in $X^{2}_{(1, \dots, 1, i)}$ by $c^{2}_{i}$ obtained by the $i$th cyclic shift of $c^{2}$. 
Continuing in this way, for each $1 \le j \le k$, we color the $k^{j}$ vertices in $X^{j}_{(1, \dots, 1, i)}$ by $c^{j}_{i}$ obtained by $i$th cyclic shift of $c^{j}$ and hence we get a coloring $c^{j+1}$ of $X^{j+1}_{(1, 1, \dots, 1)}$.
In this way, after obtaining coloring $c^{k}$ for $X^{k}_{1}$, for each $i \in S \setminus \{1\}$, we apply the same coloring for $X^{k}_{i}$.
This gives us a coloring for $S^{k+1} = V(H)$.
\par
Now we show that the above coloring is a \cnbc.
For this consider an arbitrary vertex $(b_1, \dots, b_{k+1})$ of $G$.
By above partitions, we have $(b_1, \dots, b_{k+1}) \in X^{1}_{(b_1, \dots, b_{k})}$ and since it is a rainbow set of $k$ vertices, $(b_1, \dots, b_{k+1})$ has $k-1$ neighbors of distinct colors in $X^{1}_{(b_1, \dots, b_{k})}$.
Further, the vertex $(b_1, \dots, b_{k+1})$ has $k-1$ neighbors, exactly one in each of the $k-1$ other $X^1$ sets, which are in the same $X^{2}_{(b_1, \dots, b_{k-1})}$.
Continuing in this way, for each $1 \le i \le k-1$, $(b_1, \dots, b_{k+1}) \in X^{i}_{(b_1, \dots, b_{k-i+1})}$ has $k-1$ neighbors, exactly one in each of the other $X^{i}$s, which are in $X^{i+1}_{b_1, \dots, b_{k-i}}$.
Note that, so far, the vertex $(b_1, \dots, b_{k+1})$ has exactly $k$ vertices in each of the $k-1$ color classes, different from its own color in its closed neighborhood.
Lastly, the vertex $(b_1, \dots, b_{k+1})$ has  $k-1$ neighbors, exactly one in each of the other $X^{k}$ sets.
All of these neighbors have the same color as $(b_1, \dots, b_{k+1})$, since all $X^{k}$ sets are colored under the same coloring scheme.
This ensures that the vertex $(b_1, \dots, b_{k+1})$ has an equal number of vertices in each color class in its closed neighborhood.
Therefore, we conclude that $H(k+1,k)$ is a \cnbcd graph.
This completes the base case.
\par
Now suppose $n \ge 2$.
Assume that the result is true for all hamming graphs of the form $H(kj+1, k)$ for all $ j < n$.
Consider the hamming graph \( H(kn+1, k) \).
By definition, \( H(kn+1, k) \) contains \( k \) vertex-disjoint copies of \( H(kn , k) \), each of which in turn contains \( k \) vertex-disjoint copies of \( H(kn - 1, k) \), and so on.
In particular, \( H(kn+1, k) \) contains \( k^k \) vertex-disjoint copies of \( H(kn - k+1, k) \).
By the induction hypothesis, each copy of \( H(kn - k+1, k) \) admits a \nbc, say \( c_{kn-k} \).
\par
Now, consider a specific copy of \( H(kn , k) \) in \( H(kn+1, k) \).
Within this  \( H(kn , k) \), select one copy of \( H(kn  - 1, k) \), then a copy of \( H(kn - 2, k) \), and continue down to a copy \( H(kn - k+1, k) \).
Let us assume that this copy is colored using \( c_{kn-k} \).
In the same copy of \( H(kn - k + 2, k) \), color the remaining \( k - 1 \) copies of \( H(kn - k+1, k) \) using the $k-1$ colorings \( c^2_{kn-k}, c^3_{kn-k}, \dots, c^{k}_{kn-k} \) obtained by cyclic shift of \( c_{kn-k} \).
We denote the resulting coloring of this \( H(kn - k + 2, k) \) as \( c_{kn-k+1} \).
Repeat this process: use cyclic shifts \( c^2_{kn-k+1}, c^3_{kn-k+1}, \dots, c^{k}_{kn-k+1} \) to color the remaining \( k - 1 \) copies of \( H(kn - k + 2, k) \), resulting in a coloring \( c_{kn-k+2} \) of  \( H(kn - k + 3, k) \).
Continue this recursive coloring until a coloring \( c_{kn-1} \) is obtained for \( H(kn  , k) \).
Apply this same coloring to each of the other \( k - 1 \) copies of \( H(kn  , k) \) to obtain a complete coloring of \( H(kn+1, k) \).
\par
Now consider a vertex \( v \in V(H(kn+1, k)) \).
This vertex belongs to a specific copy of \( H(kn - k+1, k) \), say $H_v$, which lies within some \( H(kn - k + 2, k) \), which in turn lies within \( H(kn - k + 3, k) \), and so on, up to some \( H(kn  , k) \).
In \(H_v =  H(kn - k+1, k) \), the vertex \( v \) has exactly one neighbor in each of the other $k-1$ color classes, as all $H(kn-k+1,k)$ copies receive one of the colorings from \( c_{kn-k} = c^1_{kn-k}, c^2_{kn-k}, \dots, c^{k}_{kn-k} \), all of which are closed neighborhood-balanced $k$-colorings.
Additionally, \( v \) has \( k - 1 \) neighbors, exactly one from each of the remaining $k-1$ copies of \( H(kn - k+1, k) \) within the same \( H(kn - k + 2, k) \), each with a distinct color differing by a cyclic shift.
In general, for each \( 1 \leq i \leq k - 1 \), the vertex \( v \) has \( k - 1 \) neighbors in the corresponding \( H(kn - k +1+ i, k) \) copies within \( H(kn - k +2+ i , k) \), where again each of these neighbors receives a distinct color via cyclic rotation.
Thus, up to this stage \( v \) has exactly \( k \) neighbors of every color except its own color.
Lastly, since the coloring \( c_{kn-1} \) is applied identically across all copies of \( H(kn, k) \) within \( H(kn+1, k) \), each vertex \( v \) has \( k - 1 \) neighbors of its own color, ensuring that $N[v]$ is equally colored.
This shows that $H(kn, k)$ is \cnbcd.
\end{proof}
\begin{table}[p]
    \setlength{\tabcolsep}{5pt}
    \centering
    \input{partitionHG}
    \caption{This table presents the partition of the vertex set of the Hamming graph $H(k+1,k)$ for $k=3$, as outlined in the proof of Theorem \ref{th: hamming graphs}.
    Each subtable illustrates the partition of the set $X_{(a)}^{3}$ for each $1 \le a \le 3$, which we refer to as a first level of partition. The second column in each subtable denotes the partition of $X_{a}^{3}$ into sets $X_{(a,b)}^{2}$ for $1 \le b \le 3$, which we refer to as the second level of partition.
    Finally, the third column in each subtable represents the partition of $X^{2}_{(a,b)}$ into sets $X^1_{(a,b,c)}$ for each $1 \le c \le 3$, which we refer to as the third level of partition.
    This table also provides a closed neighborhood-balanced $3$-coloring of $H(4,3)$.}
    \label{tab: partition of vertex set of H(4,3)}
\end{table}

\subsection{\small Non-hereditary property of the class of closed neighborhood balanced $k$-colored graphs}

A family $\mathcal{F}$ of graphs is hereditary if $G \in \mathcal{F}$ and $H$ is an induced subgraph of $G$ together imply that $H \in \mathcal{F}$.
We know that hereditary classes can be characterized by providing a list of forbidden induced subgraphs.
Indeed, a family of graphs is hereditary if and only if it has a forbidden induced subgraph characterization.
Next, we show that the class of \cnbcd graphs is not hereditary.
That is, there is no graph that is a forbidden induced subgraph for the class of \cnbcd graphs.
Note that the results of this section are already known for the particular case $k=2$ (see~\cite{collin-cnbc}).
\begin{theorem}\label{heredity}
    Every graph is an induced subgraph of a \cnbcd graph.
\end{theorem}
\begin{proof}
    Let $G$ be a graph and write $V(G)=\{ v_1, \dots, v_n\}$.
    Let $H$ be a graph with the vertex set $V(H) = \displaystyle\cup_{j=1}^{k}\{v_i^j : 1\leq i \leq n\}$ and the edge set
    \begin{align*}
    E(H) = &\Big(\displaystyle\cup_{p=1}^{k}\{v_i^pv_j^p : v_iv_j\in E(G)\}\Big) \cup \Big(\displaystyle\cup_{p,q=1, p \ne q}^{k}\{v_i^pv_j^q : v_iv_j \in E(G)\}\Big)\cup\\&\bigg(\cup_{p,q=1,p\not=q}^k\{v_i^pv_i^q;\ 1\leq i\leq n\}\bigg) .
    \end{align*}
    Note that $G$ is an induced subgraph of $H$.
    In graph $H$, color each vertex $v_i^j$ with color $j$, for all $1\leq i \leq n$, where $1\leq j \leq k$.
    Then every vertex of $H$ has an equal number of neighbors of each color in its closed neighborhood and thus $H$ is a \cnbcd graph.
\end{proof}
\begin{corollary}
   The class of \cnbcd graphs is not hereditary. 
\end{corollary}

\subsection{\small Counting Results}
For a graph $G$ and subsets $X , Y \subseteq V(G)$, we denote the set of edges joining a vertex in $X$ to a vertex in $Y$ by $E(X,Y)$.
In particular, if $X = Y$, then we write $E(X)$ instead of $E(X,X)$.
Also, we denote a subgraph induced by $X \subseteq V(G)$ by $G\left<X\right>$.
Throughout this article, unless stated otherwise, we assume that $c$ is a $k$-coloring of the graph under consideration, where $k \ge 2$, using colors $1, 2,\dots, k$.
We denote the corresponding color classes as $V_1^{c}(G), V_2^{c}(G),\dots, V_k^{c}(G)$  which form a partition of the vertex set $V(G)$ into $k$ disjoint subsets. If the coloring $c$ and the graph $G$ are clear from the context, then we simply write $V_1, V_2,\dots, V_k$.
\begin{theorem}\label{cnbccounting1}
If a graph $G$ admits a \cnbc $c$, then
   $$|E(V_i,V_j)|=\frac{2|E(G)|+|V(G)|}{k^2},$$ and $$|E(V_i,V_i)|=\frac{2|E(G)|+|V(G)|}{2k^2}-\frac{|V_i|}{2}.$$
\end{theorem}
\begin{proof}
    Consider the bipartite subgraph $H$ of $G$, induced by the set $E(V_i,V_j)$ in the bipartition $V_i\cup V_j$. Then for any $v\in V(H)$,
    \begin{equation*}
        d_H(v)=\frac{(d_G(v)+1)}{k}.
    \end{equation*}
    Therefore, counting $|E(V_i,V_j)|$ by summing degrees over each part of the bipartite graph, we get
    \begin{equation}\label{cnbceq1}
        |E(V_i,V_j)|=\sum_{v\in V_i}\frac{(d_G(v)+1)}{k}=\sum_{v\in V_j}\frac{(d_G(v)+1)}{k},
    \end{equation}
    which implies 
    \begin{equation*}
        \sum_{v\in V_i}(d_G(v)+1)=\sum_{v\in V_j}(d_G(v)+1).
    \end{equation*}
    Thus,
    \begin{equation}\label{cbnceq2}
       \sum_{v\in V_i}d_G(v)+|V_i|=\sum_{v\in V_j}d_G(v)+|V_j|. 
    \end{equation}
    Now, 
    \begin{equation*}
        \sum_{v\in V_1}d_G(v)+\dots+\sum_{v\in V_k}d_G(v)+|V_1|+\dots+|V_k|=2|E(G)|+|V(G)|.
    \end{equation*}
    Using Equations (\ref{cnbceq1}) and (\ref{cbnceq2}) , we get 
\begin{align*}
    k\bigg(\sum_{v\in V_i}d_G(v)+|V_i|\bigg)&=2|E(G)|+|V(G)|,\\
    k^2 |E(V_i,V_j)|&=2|E(G)|+|V(G)|,\\\\
    |E(V_i,V_j)|&=\frac{2|E(G)|+|V(G)|}{k^2}.
\end{align*}
This proves the first equality. For the second equality, let $G\left<V_i \right>$ be the subgraph induced by the set $V_i$. For all $v\in V_i$, we have
\begin{equation*}
    d_{G\left< V_i\right>}(v)=\frac{d_G(v)+1}{k}-1=\frac{d_G(v)+1-k}{k}.
\end{equation*}
Therefore, from Equation (\ref{cnbceq1}) and the first equality of the theorem, we obtain
\begin{align*}
    2|E(V_i,V_i)|&=\sum_{v\in V_i}d_{G\left<V_i\right>}(v)\\
    &=\sum_{v\in V_i}\frac{d_G(v)+1-k}{k}\\
    &= \sum_{v\in V_i}\bigg[\frac{d_G(v)+1-k}{k}+1-1\bigg]\\
    &=\sum_{v\in V_i}\bigg[\frac{d_G(v)+1}{k}-1\bigg]\\
    &=|E(V_i,V_j)|-|V_i|\\
    &=\frac{2|E(G)|+|V(G)|}{k^2}-|V_i|.
\end{align*}
Therefore, 
\begin{equation*}
    |E(V_i,V_i)|=\frac{2|E(G)|+|V(G)|}{2k^2}-\frac{|V_i|}{2}.
\end{equation*}
This proves the second equality and completes the proof.
\end{proof}
\begin{corollary}
    For any graph $G$ admitting a \cnbc $c$, if $|V_i|=|V_j|$, then $|E(V_i,V_i)|=|E(V_j,V_j)|.$
\end{corollary}
\begin{theorem}\label{regcountingcnbc}
    If $G$ is an $r$-regular graph admitting a \cnbc $c$, then $|V_i|=\frac{|V(G)|}{k}$, $|E(V_i,V_j)|=\frac{(r+1)|V(G)|}{k^2}$ and $|E(V_i,V_i)|=\frac{(r+1-k)|V(G)|}{2k^2}.$ 
\end{theorem}
\begin{proof}
   Each vertex colored $i$ has $\frac{r+1-k}{k}$ also colored $i$ and $\frac{r+1}{k}$ neighbors of each of the other $k-1$ colors. First, we count the edges with one endpoint colored $i$ and the other endpoint colored $j$, and we get$$
       |E(V_i,V_j)|=\sum_{v\in V_i}\frac{r+1}{k}=\bigg(\frac{r+1}{k}\bigg)|V_i|.$$
   When we count the same quantity at the endpoint colored $j$, we get $$|E(V_i,V_j)|=\bigg(\frac{r+1}{k}\bigg)|V_j|.$$
   Thus, $|V_i|=|V_j|$ and as $|V_i|+\dots+|V_k|=|V(G)|$, we have $|V_i|=\frac{|V(G)|}{k}$ and hence $$|E(V_i,V_j)|=\frac{(r+1)|V(G)|}{k^2}.$$\\
   Further using Theorem \ref{cnbccounting1} and $|V_i|$, we have
   \begin{align*}
       |E(V_i,V_i)|&=\frac{2|E(G)|+|V(G)|}{2k^2}-\frac{|V(G)|}{2k}\\
       &=\frac{r|V(G)|+|V(G)|-k|V(G)|}{2k^2}\\
       &=\frac{(r+1-k)|V(G)|}{2k^2}.
   \end{align*}
\end{proof}
\begin{corollary}\label{regcountingcor}
    If $G$ is an $r$-regular \cnbcd graph, then either $|V(G)|=0 \pmod {k^2}$ or $r\equiv (k-1) \pmod {k^2}$. 
\end{corollary}
\begin{proof}
    By Theorem \ref{regcountingcnbc}, for any \cnbc $c$ of $G$, we have $|E(V_i,V_i)|=\frac{(r+1-k)|V(G)|}{2k^2}$, which therefore must be an integer. Thus, at least one of $|V(G)|$ or $(r+1-k)$ must be a multiple of $k^2.$ 
\end{proof}

\subsection{\small Closed neighborhood balanced $k$-colored graphs having unequal sizes of color classes}
We know that for \cnbcd regular graphs, the color classes are of the same size (see Corollary \ref{regcountingcnbc}).
This need not be the case in general.
In this section, we give a way to start with a \cnbcd graph and construct a new \cnbcd graph that has fewer vertices of one color compared to the other colors.
Similar constructions for $k=2$ are shown in ~\cite{collin-cnbc}.
\begin{definition}
    Given a \cnbc{} of a graph $G$ and a vertex $z \in V(G)$ with color $k$, a \emph{$(3k-2)$-vertex addition at $z$} is defined as follows. We add vertices $u_1,u_2,\dots,u_k$, each adjacent to $z$, and for every $1 \leq i \leq k-1$ the vertex $u_i$ is adjacent to $u_{i+1},u_{i+2},\dots,u_{k-1}$. Furthermore, we introduce vertices $v_1,v_2,\dots,v_{k-1}$ and $v_1',v_2',\dots,v_{k-1}'$, all of which are adjacent to $u_k$. Among these, for each $1 \leq i \leq k-2$, the vertex $v_i$ is adjacent to $v_{i+1},v_{i+2},\dots,v_{k-1}$, and similarly $v_i'$ is adjacent to $v_{i+1}',v_{i+2}',\dots,v_{k-1}'$. Finally, we assign colors by letting $u_i$ receive color $i$ for $1 \leq i \leq k$, and both $v_i$ and $v_i'$ receive color $i$ for $1 \leq i \leq k-1$.
\end{definition}
It is straightforward to verify that a graph obtained by $(3k-2)$-vertex addition as defined above from a \cnbcd is \cnbcd.
Also, the $(3k-2)$-vertex addition adds one additional vertex of one color and two additional vertices, each of the other $k-1$ colors.
Hence, we have the following proposition.

\begin{proposition}\label{vertexaddcnbc}
    Given a \cnbc of a graph $G$, let $G'$ be the graph resulting from a $(3k-2)$-vertex addition at a vertex $z$ of $G$. Then $G'$ is a \cnbcd graph and $G'$ has one additional vertex of $z$'s color and three additional vertices of each of the other $k-1$ colors.
\end{proposition}

\begin{corollary}
    There exist \cnbcd graphs with \cnbc that have arbitrarily fewer vertices of one color than the other $k-1$ colors. Moreover, every \cnbcd graph is an induced subgraph of such a graph. 
\end{corollary}
\begin{proof}
    Let $G$ be a \cnbcd graph and fix a \cnbc $c$ of $G$. Consider a vertex having color, say $1$. By Proposition \ref{vertexaddcnbc}, a graph $G'$ obtained from $G$ by $(3k-2)$-vertex addition is \cnbcd graph and has one additional vertex of color $1$, and three additional vertices of each of the other $k-1$ colors.
    We again do $(3k-2)$-vertex addition at $G'$ (at a vertex having color $1$) to obtain another graph $G''$ that now has two additional vertices (than $G$) of color $1$, and six additional vertices (than $G$) of other $k-1$ colors.
    Repeating such $(3k-2)$-vertex additions, finally, we obtain a \cnbcd graph that has arbitrarily fewer vertices of color $1$ as compared to vertices of other $k-1$ colors.
\end{proof}
We now show an example of a graph obtained by $(3k-2)$-vertex addition to a \cnbcd graph $K_k$. This example will be referred to in the next subsection when providing counterexamples to certain theorems.

\begin{example}\label{constructionofHk}
    In this example, we construct a graph $H_k$ having order $4k-2$ that is a $(3k-2)$-vertex addition to $K_k$ and is a \cnbcd graph. Let the vertices of $K_k$ be $v_1,v_2,\dots,v_k$, and make a $(3k-2)$-vertex addition at $v_k$ with vertices $u_1,u_2,\dots,u_k$,$a_1,a_2,\dots,a_{k-1}$ and $b_1,b_2,\dots,b_{k-1}$. Since $K_k$ is a \cnbc graph, by Proposition \ref{vertexaddcnbc}, $H_k$ is a \cnbcd graph.
\end{example}

\subsection{\small Operations with \cnbcd graphs}
In this subsection, we study the \cnbcd graphs under various graph operations. The definitions of the graph operations are mentioned before stating the related result.
\begin{theorem}\label{complementcoloring}
    A coloring $c$ satisfying $|V_1|=|V_2|=\dots=|V_k|$ is a \nbc of a graph $G$ if and only if it is a \cnbc of $\overline{G}$.
\end{theorem}
\begin{proof}
Suppose $c$ is a neighborhood balanced $k$-coloring  of $G$ with 
$
|V_1|=|V_2|=\cdots=|V_k|=p.
$
Take any vertex $v \in V(G)$. By the definition of the coloring $c$, $|N_G(v)\cap V_i| = j$ for some $j$ and for all $1 \leq i \leq k$.
 
If $v$ is colored $1$, then in $\overline{G}$ the vertex $v$ has 
$p-j$ neighbors in each of the color classes $V_2,V_3,\dots,V_k$, 
and 
$p-j-1$ neighbors in $V_1$. 
Hence, for every $1\leq i \leq k$, we obtain
$
|N_{\overline{G}}[v]\cap V_i| = p-j.
$
A similar argument holds when $v$ has color $2,3,\dots,k$. Therefore, $c$ is a closed neighborhood balanced $k$-coloring of $\overline{G}$.

\medskip

Conversely, suppose $c$ is a \cnbc\ of $\overline{G}$ with 
$|V_1|=|V_2|=\cdots=|V_k|=p$.
Take any vertex $v \in V(\overline{G})$. By definition,$|N_{\overline{G}}[v]\cap V_i| = j $, for some $j$ and for all $1 \leq i \leq k$.

If $v$ has color $1$, then $v$ has $(j-1)$ neighbors of color $1$ in $\overline{G}$, which implies
$|N_G(v)\cap V_1| = p-j$.
Moreover, $v$ has $j$ neighbors in $\overline{G}$ belonging to each of the color classes $V_2,V_3,\dots,V_k$, and therefore
$|N_G(v)\cap V_i| = p-j$ for all $2 \leq i \leq k$.
A similar argument applies to vertices of color $2,3,\dots,k$. Thus $c$ is an \nbc\ of $G$.
\end{proof}

The next example shows that the hypothesis that $|V_1|=|V_2|=\dots=|V_k|$ is necessary. Recall the definition of $(2k-1)$-vertex addition used by Almeida et al. \cite{knbc} to construct neighborhood balanced graphs having unequal sizes of color classes. We shall use this operation to construct the example we require.
\begin{example}
    In this example, we construct a graph $M_k$ of order $(4k-1)$, which is a \nbcd graph and whose complement is not a \cnbcd graph. Start with $\overline{K_{2k}}$ that has a \nbc that has two vertices of each color. Make a $(2k-1)$-vertex addition at one of these $2k$-vertices. As shown in \cite{knbc}, the resulting graph $M_k$ is a \nbcd graph, and has a \nbc with three vertices of one color and four vertices each of the other $(k-1)$ colors. Each vertex of $M_k$ has a degree multiple of $k$ and $(4k-1)$ is not a multiple of $k$, so each vertex in $\overline{M_k}$ has a degree not congruent to $-1\pmod{k}$ and hence $\overline{M_k}$ is not a \cnbcd graph.
\end{example}
The strong product of graphs $G$ and $H$, denoted by $G \boxtimes H$, is the graph with vertex set $V(G)\times V(H)$ where two vertices $(g,h)$ and $(g',h')$ are adjacent in $G \boxtimes H$ if and only if $(1)$ $g=g'$ and $hh'\in E(H)$ or $(2)$ $h=h'$ and $gg'\in E(G)$ or $(3)$ $gg'\in E(G)$ and $hh'\in E(H)$. We call these the first, second, and third types of edges.
\begin{theorem}
    If $G$ is a \cnbcd graph, and $H$ is any graph, then $G \boxtimes H$ is a \cnbcd graph.
\end{theorem}
\begin{proof}
   For each $h \in V(H)$, define
\[
W_h = \{(g,h) : g \in V(G)\}.
\]
The set $W_h$ induces a copy of $G$ as a subgraph of $G \boxtimes H$. Color each $W_h$ using the given \cnbc\ of $G$. We claim that this yields a \cnbc\ of $G \boxtimes H$.

Type one neighbors lie in $W_h$. Since $W_h$ is colored with a \cnbc\ of $G$, the closed neighborhood $N_{W_h}[(g,h)]$ is balanced.  

For type two and three neighbors: for each $h' \in V(H)$ with $hh' \in E(H)$, the neighbors of $(g,h)$ inside $W_{h'}$ are precisely those in the closed neighborhood $N_{W_{h'}}[(g',h')]$, corresponding to types two and three. As $W_{h'}$ was colored using a \cnbc\ of $G$, this neighborhood is colored equally using each of the $k$ colors.

Thus, for every $h' \sim h$, all neighbors in $W_{h'}$ are colored equally using each of the $k$ colors. Combining this with the neighbors from $W_h$, it follows that the entire closed neighborhood $N_{G \boxtimes H}[(g,h)]$ is colored equally using each of the $k$ colors.

Since $(g,h)$ was arbitrary, the resulting coloring is a \cnbc\ of $G \boxtimes H$.
\end{proof}
The cartesian product of graphs $G$ and $H$, denoted by $G\square H$, is a graph with vertex set $V(G)\times V(H)$ where two vertices $(g,h)$ and $(g',h')$ are adjacent in $G\square H$ if and only if $g=g'$ and $hh'\in E(H)$  or $h=h'$ and $gg'\in E(G)$.
\begin{theorem}
    If $G$ is a \cnbcd graph, then $G\square K_2$ is a \nbcd graph.
\end{theorem}
\begin{proof}
    Let $V(K_2)=\{h_0,h_1\}$. Then the vertices of $G \square K_2$ are of the form $(g,h_i)$, where $g \in V(G)$ and $i \in \{0,1\}$.
    
   Let $c$ be a \cnbc\ of $G$, and let $V(G) = V_1 \cup V_2 \cup \cdots \cup V_k$ denote its color classes. Define a coloring of $G \square K_2$ as follows:  
\[
c((g,h_i)) = j \quad \text{if } g \in V_j.
\] Consider the vertex $(g,h_i)$. There are an equal number of vertices of every color in the set consisting of $(g,h_i)$ together with its neighbors of the form $(g',h_i)$ since $c$ is a \cnbc of $G$. In $G\square K_2$ vertex $(g,h_i)$ has one additional neighbor $(g,h_{1-i})$, which has the same color as $(g,h_i)$. Thus, the neighborhood of $(g,h_i)$ has an equal number of neighbors of each of the $k$ colors, and hence our coloring is a \nbc of $G\square K_2$.
\end{proof}

The lexicographic product of graphs $G$ and $H$, denoted by $G[H]$, is a graph with vertex set $V(G)\times V(H)$ where two vertices $(g,h)$ and $(g',h')$ are adjacent in $G\square H$ if and only if $gg'\in E(G)$  or $g=g'$ and $hh'\in E(H)$. 
\begin{theorem}\label{lexicnbc}
    Let $G$ be any graph and suppose that $H$ is a graph that admits a \cnbc $c$ such that $|V_1^c(H)|=|V_2^c(H)|=\dots=|V_k^c(H)|$, then $G[H]$ is a \cnbcd graph.
\end{theorem}
\begin{proof}
By Theorem~\ref{lexiknbc}, for any two graphs $W_1$ and $W_2$ where $W_2$ is a \nbcd{} graph with a neighborhood balanced $k$-coloring $g$ satisfying 
$$|V^g_1(W_2)| = |V^g_2(W_2)| = \cdots = |V^g_k(W_2)|,$$
the lexicographic product $W_1[W_2]$ is also a \nbcd{} graph.

Let $W_1 = \overline{G}$ and $W_2 = \overline{H}$.  As $|V_1^c(H)|=|V_2^c(H)|=\dots=|V_k^c(H)|$, by Theorem~\ref{complementcoloring}, $\overline{H}$ is a \nbcd{} graph with the same neighborhood balanced coloring $c$ as $H$, satisfying $|V_1^c(\overline{H})|=|V_2^c(\overline{H})|=\dots=|V_k^c(\overline{H})|$. Hence, $\overline{G}[\overline{H}]$ is a \nbcd{} graph with a coloring $h$ satisfying$$
|V^h_1(\overline{G}[\overline{H}])| = |V^h_2(\overline{G}[\overline{H}])| = \cdots = |V^h_k(\overline{G}[\overline{H}])|,$$
a fact that follows directly from the proof of Theorem~\ref{lexiknbc}.

Finally, since
\[
\overline{\overline{G}[\overline{H}]} \;=\; G[H],
\]
it follows by Theorem~\ref{complementcoloring} that $G[H]$ is a \cnbcd{} graph.
\end{proof}
We now show with an example that the condition of equal sizes of color classes of the graph $H$ in Theorem \ref{lexicnbc} is necessary. We show that there exists a graph $G$ and a \cnbcd{} graph $H$ such that $G[H]$ is not a \cnbcd{} graph.  
    \begin{example}
Let $G = K_k$ and let $H = H_k$ be the graph defined in Example~\ref{constructionofHk}.  
Then the lexicographic product $K_k[H_k]$ is the join of $k$ copies of $H_k$.  

Denote the vertices in the $r^{\text{th}}$ copy of $H_k$ by
$
\{v_1^r, v_2^r, \dots, v_k^r\} \;\cup\; 
\{u_1^r, u_2^r, \dots, u_k^r\} \;\cup\; 
\{a_1^r, a_2^r, \dots, a_{k-1}^r\} \;\cup\; 
\{b_1^r, b_2^r, \dots, b_{k-1}^r\}.
$
Now consider the vertex $u_k^1$ in the first copy of $H_k$.  
Its degree is
$
\deg(u_k^1) \;=\; (2k-1) + (k-1)(4k-2) \;=\; 4k^2 - 4k + 1.
$
Since
$
4k^2 - 4k + 1 \not\equiv -1 \pmod{k},
$
it follows that $K_k[H_k]$ is not a \cnbcd{} graph.
\end{example}
The join of graphs $G$ and $H$, denoted by $G\vee H$, is a graph having vertex set $V(G)\cup V(H)$ and edge set $E(H)\cup E(H)\cup \{gh; g\in V(G), h\in V(H)\}$.
\begin{theorem}\label{joincnbc}
Let $G$ be a \cnbcd graph with coloring $c$ such that $|V_1^c(G)|=|V_2^c(G)|=\dots=|V_k^c(G)|$ and $H$ be a \cnbcd graph with coloring $c'$ such that $|V_1^{c'}(H)|=|V_2^{c'}(H)|=\dots=|V_k^{c'}(H)|$. Then $G \vee H$ is a \cnbcd graph.
\end{theorem}
\begin{proof}
In the join $G \vee H$, color each vertex of the subgraph $G$ using the coloring $c$ and each vertex of the subgraph $H$ using the coloring $c'$.  
For this coloring of $G\vee H$, the number of vertices in each color class is equal, since
\[
|V^c_1(G)| = |V^c_2(G)| = \cdots = |V^c_k(G)| 
\quad \text{and} \quad
|V^{c'}_1(H)| = |V^{c'}_2(H)| = \cdots = |V^{c'}_k(H)|.
\]
Hence, when $G$ and $H$ are joined, each vertex of $G\vee H$ receives the same number of additional vertices, and the proof follows.
\end{proof}
We now show that the condition on color classes being of equal size in Theorem \ref{joincnbc} is necessary. We show that there exist two \cnbcd graphs whose join is not a \cnbcd graph.

    \begin{example}
Consider the join $K_k \vee H_k$, where $H_k$ is the graph defined in Example~\ref{constructionofHk}.  
Suppose, for a contradiction, that $K_k \vee H_k$ is a \cnbcd{} graph, and let $c$ be a corresponding \cnbc.  

Note that the closed neighborhood of any vertex in $K_k$ is the entire vertex set $V(K_k \vee H_k)$.  
Therefore, the coloring $c$ must satisfy
$
|V_1^c| = |V_2^c| = \cdots = |V_k^c|.
$

Now consider the complement. Since
$
\overline{K_k \vee H_k} \;=\; \overline{K_k} \vee \overline{H_k},
$
Theorem~\ref{complementcoloring} implies that $c$ is a neighborhood balanced coloring of $\overline{K_k} \vee \overline{H_k}$.  
But in this graph, the only edges lie within $\overline{H_k}$.  
Thus, $c$ restricts to a neighborhood balanced coloring of $\overline{H_k}$.
However, in $\overline{H_k}$ the degree of the vertex $u_k$ is
$
\deg_{\overline{H_k}}(u_k) = 2k - 1,
$
which is not congruent to $0 \pmod{k}$.  
This contradicts the requirement for a neighborhood balanced coloring (refer to Lemma \ref{nbclemma}).  

Hence, $K_k \vee H_k$ cannot be a \cnbcd{} graph.
\end{example}

The direct product of graphs $G$ and $H$, denoted by $G\times H$, is a graph with vertex set $V(G)\times V(H)$ where two vertices $(g,h)$ and $(g',h')$ are adjacent if and only if $gg'\in E(G)$ and $hh'\in E(H)$.
\begin{theorem}
    If $G$ and $H$ are both \cnbcd graphs, then $G\times H$ is not a \cnbcd graph.
\end{theorem}
\begin{proof}
    As $G$ is a \cnbcd graph, the degree of every vertex of $G$ is congruent to $-1 \pmod k$. The same is true with respect to the degree of every vertex of $H$. Let $(u,v)\in V(G\times H)$. As $u\in V(G)$, $v\in V(H)$, we have $\deg_G(u)\equiv -1 \pmod k$ and $\deg_H(v)\equiv -1 \pmod k$. As $\deg_{G\times H}(u,v)=\deg_G(u)\times \deg_G(v)=1\pmod{k}$, $G\times H$ is not a \cnbcd graph.
\end{proof}
\begin{theorem}
    If $G$ is a \cnbcd graph and $H$ is a \nbcd graph, then $G\ \square\ H$ is a \cnbcd graph.
\end{theorem}
    \begin{proof}
Let $G$ admit a closed neighborhood balanced coloring 
$c_G : V(G) \to \{1,2,\dots,k\}$ and let $H$ admit a neighborhood balanced coloring 
$c_H : V(H) \to \{1,2,\dots,k\}$. Define
\[
c : V(G) \times V(H) \to \{1,2,\dots,k\}
\quad\text{by}\quad
c(g,h) = \bigl(c_G(g) + c_H(h) - 1\bigr) \bmod k + 1.
\]
We show that $c$ is a closed neighborhood balanced coloring of $G \square H$.

Fix a vertex $(g,h) \in V(G \square H)$. Without loss of generality, assume $c_G(g) = k$. 
Since $c_H$ is a neighborhood balanced coloring of $H$, there exists an integer $j$ such that 
among the neighbors of $(g,h)$ of the form $(g,h')$, the vertex $(g,h)$ has exactly $j$ 
neighbors of each of the $k$ colors.

All other neighbors of $(g,h)$ are of the form $(g',h)$. Because $c_G$ is a closed neighborhood 
balanced coloring of $G$, there exists an integer $i$ such that $g$ has exactly $i$ neighbors 
of color $k$ and $i+1$ neighbors of each of the other $k-1$ colors in its closed neighborhood.

Now let $c_H(h) = q$. Then by definition of $c$,
$c(g,h) = \bigl(k + q - 1\bigr) \bmod k + 1 = q$, so $(g,h)$ receives color $q$. Among neighbors of the form $(g,h')$, it has $j$ neighbors of color $q$ 
and $j$ neighbors of each other color. Among neighbors of the form $(g',h)$, it has $i$ neighbors 
of color $q$ and $i+1$ neighbors of each of the other $k-1$ colors.

Hence $(g,h)$ has exactly $i + j$ neighbors of its own color in its closed neighborhood and 
$i + j + 1$ neighbors of each of the remaining $k-1$ colors. Therefore $c$ is a closed neighborhood 
balanced coloring of $G \square H$.
\end{proof}

\begin{theorem}
    If $G$ is a \cnbcd graph, then $G\times K_2$ is a \cnbcd graph.
\end{theorem}

\begin{proof}
     Let $V(K_2)=\{h_0,h_1\}$. Then the vertices of $G \times K_2$ are of the form $(g,h_i)$, where $g \in V(G)$ and $i \in \{0,1\}$.
    
   Let $c$ be a \cnbc\ of $G$, and let $V(G) = V_1 \cup V_2 \cup \cdots \cup V_k$ denote its color classes. Define a coloring of $G \times K_2$ as follows:  
\[
c((g,h_i)) = j \quad \text{if } g \in V_j.
\] 
Now consider a vertex $(g,h_i) \in V(G \times K_2)$.  
Its neighbors are precisely the vertices of the form $(g',h_{1-i})$, where $g'$ is a neighbor of $g$ in $G$.  
Since $c$ is a \cnbc of $G$, the vertex $g$ has an equal number of neighbors in each color class of $G$.  
Therefore, the closed neighborhood of $(g,h_i)$ in $G \times K_2$ also contains an equal number of vertices of every color.  

Hence, this coloring is a closed neighborhood balanced coloring of $G \times K_2$, and so it is a \cnbcd graph.
\end{proof}

%% file: partitionHG.tex
\begin{subtable}[t]{0.48\textwidth}
  \centering
  \caption{Partition of $X_{(1)}^{3}$}
  \renewcommand{\arraystretch}{1.5}
  \resizebox{\linewidth}{!}{%
  \begin{tabular}{|c|c|c|c|c|c|}
  \hline
  \multirow{9}{*}{$X_{(1)}^{3}$} & \multirow{3}{*}{$X_{(1,1)}^{2}$} & $X_{(1,1,1)}^{1}$ &\cellcolor{red!40} $(1,1,1,1)$ & \cellcolor{blue!40}$(1,1,1,2)$ & \cellcolor{green!40}$(1,1,1,3)$ \\ \cline{3-6}
                                  &                                  & $X_{(1,1,2)}^{1}$ &\cellcolor{blue!40} $(1,1,2,1)$ &\cellcolor{green!40} $(1,1,2,2)$ &\cellcolor{red!40} $(1,1,2,3)$ \\ \cline{3-6}
                                  &                                  & $X_{(1,1,3)}^{1}$ & \cellcolor{green!40}$(1,1,3,1)$ & \cellcolor{red!40} $(1,1,3,2)$ & \cellcolor{blue!40} $(1,1,3,3)$ \\ \cline{2-6}
                                  & \multirow{3}{*}{$X_{(1,2)}^{2}$} & $X_{(1,2,1)}^{1}$ & \cellcolor{blue!40}$(1,2,1,1)$ & \cellcolor{green!40}$(1,2,1,2)$ &\cellcolor{red!40} $(1,2,1,3)$ \\ \cline{3-6}
                                  &                                  & $X_{(1,2,2)}^{1}$ & \cellcolor{green!40}$(1,2,2,1)$ & \cellcolor{red!40} $(1,2,2,2)$ & \cellcolor{blue!40} $(1,2,2,3)$ \\ \cline{3-6}
                                  &                                  & $X_{(1,2,3)}^{1}$ & \cellcolor{red!40}$(1,2,3,1)$ & \cellcolor{blue!40} $(1,2,3,2)$ & \cellcolor{green!40} $(1,2,3,3)$ \\ \cline{2-6}
                                  & \multirow{3}{*}{$X_{(1,3)}^{2}$} & $X_{(1,3,1)}^{1}$ & \cellcolor{green!40}$(1,3,1,1)$ & \cellcolor{red!40} $(1,3,1,2)$ & \cellcolor{blue!40} $(1,3,1,3)$ \\ \cline{3-6}
                                  &                                  & $X_{(1,3,2)}^{1}$ & \cellcolor{red!40} $(1,3,2,1)$ & \cellcolor{blue!40}$(1,3,2,2)$ &\cellcolor{green!40} $(1,3,2,3)$ \\ \cline{3-6}
                                  &                                  & $X_{(1,3,3)}^{1}$ & \cellcolor{blue!40}$(1,3,3,1)$ & \cellcolor{green!40}$(1,3,3,2)$ & \cellcolor{red!40}$(1,3,3,3)$ \\ \cline{2-6}\hline
  \end{tabular}%
  }% end resizebox
\end{subtable}%\
\hfill
\begin{subtable}[t]{0.48\textwidth}
  \centering
  \caption{Partition of $X_{(2)}^{3}$}
  \renewcommand{\arraystretch}{1.5}
  \resizebox{\linewidth}{!}{%
  \begin{tabular}{|c|c|c|c|c|c|}
  \hline
  \multirow{9}{*}{$X_{(2)}^{3}$} & \multirow{3}{*}{$X_{(2,1)}^{2}$} & $X_{(2,1,1)}^{1}$ &\cellcolor{red!40} $(2,1,1,1)$ & \cellcolor{blue!40}$(2,1,1,2)$ & \cellcolor{green!40}$(2,1,1,3)$ \\ \cline{3-6}
                                  &                                  & $X_{(2,1,2)}^{1}$ &\cellcolor{blue!40} $(2,1,2,1)$ &\cellcolor{green!40} $(2,1,2,2)$ &\cellcolor{yellow!40} $(2,1,2,3)$ \\ \cline{3-6}
                                  &                                  & $X_{(2,1,3)}^{1}$ & \cellcolor{green!40}$(2,1,3,1)$ & \cellcolor{yellow!40} $(2,1,3,2)$ & \cellcolor{red!40} $(2,1,3,3)$ \\ \cline{2-6}
                                  & \multirow{3}{*}{$X_{(2,2)}^{2}$} & $X_{(2,2,1)}^{1}$ & \cellcolor{blue!40}$(2,2,1,1)$ & \cellcolor{green!40}$(2,2,1,2)$ &\cellcolor{yellow!40} $(2,2,1,3)$ \\ \cline{3-6}
                                  &                                  & $X_{(2,2,2)}^{1}$ & \cellcolor{green!40}$(2,2,2,1)$ & \cellcolor{yellow!40} $(2,2,2,2)$ & \cellcolor{red!40} $(2,2,2,3)$ \\ \cline{3-6}
                                  &                                  & $X_{(2,2,3)}^{1}$ & \cellcolor{yellow!40}$(2,2,3,1)$ & \cellcolor{red!40} $(2,2,3,2)$ & \cellcolor{blue!40} $(2,2,3,3)$ \\ \cline{2-6}
                                  & \multirow{3}{*}{$X_{(2,3)}^{2}$} & $X_{(2,3,1)}^{1}$ & \cellcolor{green!40}$(2,3,1,1)$ & \cellcolor{yellow!40} $(2,3,1,2)$ & \cellcolor{red!40} $(2,3,1,3)$ \\ \cline{3-6}
                                  &                                  & $X_{(2,3,2)}^{1}$ & \cellcolor{yellow!40} $(2,3,2,1)$ & \cellcolor{red!40}$(2,3,2,2)$ &\cellcolor{blue!40} $(2,3,2,3)$ \\ \cline{3-6}
                                  &                                  & $X_{(2,3,3)}^{1}$ & \cellcolor{red!40}$(2,3,3,1)$ & \cellcolor{blue!40}$(2,3,3,2)$ & \cellcolor{green!40}$(2,3,3,3)$ \\ \cline{2-6}
 \hline
  \end{tabular}%
  }% end resizebox
\end{subtable}%
\hfill
\begin{subtable}[t]{0.48\textwidth}
  \centering
  \caption{Partition of $X_{(1)}^{3}$}
  \renewcommand{\arraystretch}{1.5}
  \resizebox{\linewidth}{!}{%
  \begin{tabular}{|c|c|c|c|c|c|}
  \hline
  \multirow{9}{*}{$X_{(3)}^{3}$} & \multirow{3}{*}{$X_{(3,1)}^{2}$} & $X_{(3,1,1)}^{1}$ &\cellcolor{red!40} $(3,1,1,1)$ & \cellcolor{blue!40}$(3,1,1,2)$ & \cellcolor{green!40}$(3,1,1,3)$ \\ \cline{3-6}
                                  &                                  & $X_{(3,1,2)}^{1}$ &\cellcolor{blue!40} $(3,1,2,1)$ &\cellcolor{green!40} $(3,1,2,2)$ &\cellcolor{yellow!40} $(3,1,2,3)$ \\ \cline{3-6}
                                  &                                  & $X_{(3,1,3)}^{1}$ & \cellcolor{green!40}$(3,1,3,1)$ & \cellcolor{yellow!40} $(3,1,3,2)$ & \cellcolor{red!40} $(3,1,3,3)$ \\ \cline{2-6}
                                  & \multirow{3}{*}{$X_{(3,2)}^{2}$} & $X_{(3,2,1)}^{1}$ & \cellcolor{blue!40}$(3,2,1,1)$ & \cellcolor{green!40}$(3,2,1,2)$ &\cellcolor{yellow!40} $(3,2,1,3)$ \\ \cline{3-6}
                                  &                                  & $X_{(3,2,2)}^{1}$ & \cellcolor{green!40}$(3,2,2,1)$ & \cellcolor{yellow!40} $(3,2,2,2)$ & \cellcolor{red!40} $(3,2,2,3)$ \\ \cline{3-6}
                                  &                                  & $X_{(3,2,3)}^{1}$ & \cellcolor{yellow!40}$(3,2,3,1)$ & \cellcolor{red!40} $(3,2,3,2)$ & \cellcolor{blue!40} $(3,2,3,3)$ \\ \cline{2-6}
                                  & \multirow{3}{*}{$X_{(3,3)}^{2}$} & $X_{(3,3,1)}^{1}$ & \cellcolor{green!40}$(3,3,1,1)$ & \cellcolor{yellow!40} $(3,3,1,2)$ & \cellcolor{red!40} $(3,3,1,3)$ \\ \cline{3-6}
                                  &                                  & $X_{(3,3,2)}^{1}$ & \cellcolor{yellow!40} $(3,3,2,1)$ & \cellcolor{red!40}$(3,3,2,2)$ &\cellcolor{blue!40} $(3,3,2,3)$ \\ \cline{3-6}
                                  &                                  & $X_{(1,3,3)}^{1}$ & \cellcolor{red!40}$(1,3,3,1)$ & \cellcolor{blue!40}$(1,3,3,2)$ & \cellcolor{green!40}$(1,3,3,3)$ \\ \cline{2-6}
                                  \hline
  \end{tabular}%
  }% end resizebox
\end{subtable}%

%% file: Hardness_k_geq3.tex
We formally state the decision problem whether a given graph $G$ admits a \cnbc for a given integer $k \ge3 $.

\vspace{0.2cm}
\noindent\fbox{%
  \begin{minipage}{\linewidth}
    {\textbf{\textsc{Closed Neighborhood Balanced $k$-Coloring ($k$-CNBC)}}}\\
    \textbf{Input:} A graph $G$ and a positive integer $k\geq 3$. \hfill\\
    \textbf{Question:} Is there a vertex coloring of $G$ using $k$-colors such that every vertex has an equal number of vertices of each color in its closed neighborhood? 
  \end{minipage}
}

\vspace{0.3cm}

Our reduction relies on the $k$-proper coloring problem, whose \NP-completeness is established in \cite{karp}.

\vspace{0.2cm}
\noindent\fbox{
  \begin{minipage}{\linewidth}
    {\textbf{\textsc{$k$-Proper Coloring ($k$-PC)}}} \hfill\\
    \textbf{Input:} A graph $G$ and a positive integer $k\geq 3$.\\[3pt]
    \textbf{Question:} Is there a coloring of $G$ using $k$-colors such that adjacent vertices are not colored using the same color.
  \end{minipage}
}

\vspace{0.3cm}
Our reduction uses two gadgets, an edge clique and a padding gadget, both of which are described ahead. An \emph{edge clique} is simply a clique of order $k-2$.
\begin{definition}
A \emph{padding gadget} is a graph \( P \) of order \( 2k - 1 \) with vertex set
\[
V(P) = \{ c \} \cup \{ v_1, v_2, \dots, v_{k-1} \} \cup \{ u_1, u_2, \dots, u_{k-1} \}.
\]
The subset \( \{ c, v_1, v_2, \dots, v_{k-1} \} \) induces a \( k \)-clique, and similarly, 
\( \{ c, u_1, u_2, \dots, u_{k-1} \} \) induces another \( k \)-clique. 
The vertex \( c \) is referred to as the \emph{central vertex} of the padding gadget.
\end{definition}
Given a graph $G$ and a positive integer $k\geq 3$ that is properly colored using $k$ colors, we construct an instance $(G,k)$ of $k$-CNBC as follows.\par
For each edge $e=uv$ of $G$, introduce an edge clique. Join all the $k-2$ vertices of the edge clique to both the end vertices of $e$. Further, for each vertex $v$ of $G$, add $d(v)-1$ padding gadgets. We shall call the graph so constructed $G'$. 
\begin{observation}\label{obs:edge-clique}
   If the graph $G'$ admits a $k$-CNBC, then the $k - 2$ vertices of any edge--clique receive colors that are distinct from those assigned to the end vertices of the edge to which the edge--clique is attached.  
\end{observation}
\begin{theorem} \label{th:cnbc-npcomplete}
The $k$-\textsc{Closed Neighborhood Balanced Coloring} (\textsc{$k$-CNBC}) problem is \NP-complete for any integer $k \ge 3$.
\end{theorem}

\begin{proof}
We give a reduction from the \textsc{$k$-Proper Coloring} (\textsc{$k$-PC}) problem.  
Let $(G,k)$ be an instance of \textsc{$k$-PC}. We construct an equivalent instance $(G',k)$ of \textsc{$k$-CNBC} as described earlier. Note that $G'$ can be constructed in polynomial time. 

Since, given a vertex coloring of $G$, we can verify in polynomial time whether it satisfies the closed neighborhood balanced condition, the \textsc{$k$-CNBC} problem belongs to the class \NP.

We now show that $(G,k)$ is a \textsc{Yes}-instance of \textsc{$k$-PC} if and only if $(G',k)$ is a \textsc{Yes}-instance of \textsc{$k$-CNBC}.

\medskip
\noindent\textbf{($\Rightarrow$)} Suppose $(G,k)$ is a \textsc{Yes}-instance of \textsc{$k$-PC}.  
Let $c : V(G) \rightarrow \{1,2,\dots,k\}$ be a proper $k$-coloring of $G$.  
We define a coloring $c'$ of $G'$ as follows:
\begin{enumerate}[noitemsep]
    \item Assign each vertex of $G$ in $G'$ the same color as under $c$.
    \item For every edge $e = uv$ of $G$, color the $k - 2$ vertices of the corresponding edge clique using the $k - 2$ colors not assigned to $u$ and $v$, in any order.
    \item For each vertex $v \in V(G)$, color the central vertex of each padding gadget attached to $v$ with the same color as $v$.  
    Within each padding gadget, assign the remaining $k - 1$ colors to the other vertices of each of the two $k$-cliques so that every color from $\{1,2,\dots,k\}$ appears exactly once in each clique.
\end{enumerate}

We now prove that $c'$ is a valid \textsc{$k$-CNBC} coloring of $G'$.

\begin{description}
    \item[Vertices within padding gadgets:]
    By construction, the central vertex $c$ of each padding gadget has two vertices of each color in its closed neighborhood $N[c]$.  
    Similarly, each non-central vertex in either clique has exactly one vertex of each color in its closed neighborhood.
    
    \item[Vertices in edge cliques:]  
    From step (2), each edge clique consists of vertices colored with the $k - 2$ colors distinct from those of its adjacent endpoints $u$ and $v$.  
    Each such vertex is adjacent to $u$, $v$, and all other $(k - 3)$ vertices of the clique, giving it one vertex of each color in its closed neighborhood.

    \item[Original vertices of $G$:]  
    Consider a vertex $v \in V(G)$.  
    For each padding gadget attached to $v$, the central vertex is colored the same as $v$.  
    Hence, $v$ has exactly $d(v)$ neighbors of its own color—one from each padding gadget (since there are $d(v) - 1$ such gadgets) and itself.  
    Moreover, for every neighbor $u$ of $v$ in $G$, the edge clique corresponding to $uv$ contributes one vertex of each of the remaining $k - 2$ colors (excluding those of $u$ and $v$).  
    Therefore, $v$ has exactly $d(v)$ neighbors of each of the $k - 1$ colors distinct from its own color.  
    Consequently, every vertex of $G$ has the same number of vertices of each color in its closed neighborhood.
\end{description}

Hence, the coloring $c'$ satisfies the closed neighborhood balanced condition, and $(G',k)$ is a \textsc{Yes}-instance of \textsc{$k$-CNBC}.

\noindent\textbf{($\Leftarrow$)}  
Suppose $(G',k)$ is a \textsc{Yes}-instance of \textsc{$k$-CNBC}.  
Let $c' : V(G') \rightarrow \{1,2,\dots,k\}$ be a closed neighborhood balanced $k$-coloring of $G'$.  

Consider any edge $e = uv$ of the original graph $G$.  
In $G'$, to both the vertices of the edge $e$ we add an edge clique of order $k - 2$, where every vertex of the clique is adjacent to both $u$ and $v$.  
Thus, the subgraph induced by $\{u,v\}$ together with this edge clique forms a complete graph on $k$ vertices.

By Observation~\ref{obs:edge-clique}, in any valid $k$-CNBC of $G'$, the $k - 2$ vertices of each edge clique must receive colors distinct from those of its endpoints $u$ and $v$.  
If $u$ and $v$ were assigned the same color under $c'$, then this $k$-clique would contain only $k - 1$ distinct colors, violating the balanced condition within the closed neighborhoods of the clique vertices.  
Hence, $u$ and $v$ must receive distinct colors.

Therefore, the restriction of $c'$ to the vertex set $V(G)$ defines a proper $k$-coloring of $G$.  
Consequently, $(G,k)$ is a \textsc{Yes}-instance of \textsc{$k$-PC}.
\end{proof}

%% file: conclusion.tex
In this article, we defined the concept of closed neighborhood balanced $k$-coloring of graphs, which is a generalization of closed neighborhood balanced $2$-coloring of graphs introduced by Collins et al. \cite{cnbc}. Initially, we gave some characteristics of graphs that admit such a coloring. Then we studied \cnbcd graphs can be constructed under various graph operations. We also presented several regular \cnbcd graphs and showed that the class of \cnbcd graphs is not hereditary. In section \ref{sec:Hardness,kgeq3}, we showed that the decision problem of checking whether a given graph $G$ admits a \cnbc for $k\geq 3$ is \textbf{NP}-hard.\\
We now pose the following problems related to our work.
\begin{enumerate}
    \item Characterize the regular graphs that admit a \cnbc.
    \item Is the problem of determining whether a graph admits \cnbc for $k=2$ an \textbf{NP}-complete problem?
\end{enumerate}

% \vspace{1cm}
% \noindent \textbf{\large Acknowledgment:}\\
% The second and fourth Authors are thankful to  National Board for Higher Mathematics(NBHM), Govt. of India, for the financial support to carry out research under the project Ref. No.: 02011/21/2025/NBHM-RP/RD-II/9821.